\documentclass[12pt]{article}
\usepackage{amsfonts}
\usepackage{bbm}
\usepackage{amscd}
\usepackage{amsmath}
\usepackage{amssymb}
\usepackage{amsthm}
\usepackage{bbm}
\usepackage{CJK}
\usepackage{fancyhdr}
\usepackage{hyperref}
\usepackage{indentfirst}
\usepackage{latexsym}
\usepackage{mathrsfs}
\usepackage{color}

\usepackage[top=1in,bottom=1in,left=1.25in,right=1.25in]{geometry}
\def\<{\langle}
\def\>{\rangle}

\def\d{\delta}
\def\D{\Delta}

\def\g{\gamma}

\def\l{\lambda}
\def\ra{\rightarrow}

\def\o{\otimes}

\def\v{\varepsilon}
\def\vp{\varphi}

\def\<{\langle}
\def\>{\rangle}

\theoremstyle{definition}
\newtheorem{definition}{Definition}[section]
\newtheorem{remark}[definition]{Remark}
\newtheorem{example}[definition]{Example}

\theoremstyle{plain}
\newtheorem{theorem}[definition]{Theorem}
\newtheorem{proposition}[definition]{Proposition}
\newtheorem{lemma}[definition]{Lemma}
\newtheorem{corollary}[definition]{Corollary}

\numberwithin{equation}{section}

\begin{document}

\title{\bf The cointegral theory of weak multiplier Hopf algebras}
\date{}
\author{{\bf  Nan Zhou\footnote{Corresponding author. Zhejiang Shuren University, Hangzhou 310015, Zhejiang, China. E-mail: zhounan0805@163.com }
and Tao Yang\footnote{College of Science, Nanjing Agricultural University, Nanjing 210095, Jiangsu, China. E-mail: tao.yang@njau.edu.cn}}\\
\textbf{}
}

 \maketitle

\noindent
{\bf Abstract}

In this paper, we introduce and study the notion of cointegrals in a weak multiplier Hopf algebras $(A, \Delta)$.
 A cointegral is a non-zero element $h$ in the multiplier algebra $M(A)$ such that $ah=\v_t(a)h$ for any $a\in A$.
When $A$ has a faithful set of cointegrals (now we call $A$ of {\it discrete type}), we give a sufficient and necessary condition for existence of integrals on $A$.
Then we consider a special case, i.e., $A$ has a single faithful cointegral, and we obtain more better results, such as $A$ is Frobenius, quasi-Frobenius, et al.
Moreover when an algebraic quantum groupoid $A$  has a faithful cointegral, then the dual $\widehat{A}$ must be weak Hopf algebra.
In the end, we investigate when $A$ has a cointegral and study relation between compact and discrete type.




\section{Introduction}

Weak multiplier Hopf algebra, introduced by A. Van Daele and S. Wang in the literatures \cite{VW1,VW2,VW3,VW4},
 makes a further progress in Pontryagin duality theorem on the purely algebraic level.
 The motivation example is the algebra $A$ of complex functions with finite support on a groupoid $G$.
 The product on $A$ is pointwise multiplication and the coproduct is defined as
$$\D(f)(p, q)=\left\{
       \begin{array}{ll}
         f(pq), & \hbox{if $pq$ is defined;} \\
         0, & \hbox{otherwise,}
       \end{array}
     \right.
$$
where $p, q\in G, f\in A$.
When $G$ is a group, then $A$ degenerates to a multiplier Hopf algebra; and while $G$ is a finite groupoid, $A$ becomes a weak Hopf algebra.
From this point of view, the weak multiplier Hopf algebra  generalizes the notions of weak Hopf algebras(\cite{BNS, BS, DN}) and multiplier Hopf algebras(\cite{VD2}) respectively.

Let $(A, \Delta)$ be a weak multiplier Hopf algebra with a bijective antipode, i.e., it is regular. Assume that there exists a faithful set $\int$ of integrals on $A$,
then by the paper \cite{VW4} the space of linear functionals $\{\vp(a\cdot )|a\in A, \vp\in \int\}$ can be made into a regular weak multiplier Hopf algebra, which is called the dual of $A$ and denoted by $\widehat{A}$.
Moreover, the authors constructed the dual of $\widehat{A}$ and get that $A\cong \widehat{\widehat{A}}$, i.e., the famous Pontryagin duality holds. In this case, a regular weak multiplier Hopf algebra with a faithful set $\int$ of integrals is called  an algebraic quantum groupoid.

In this paper, we will consider the cointegrals in a weak multiplier Hopf algebra.
A cointegral is a non-zero element in the multiplier algebra $M(A)$ satisfying certain conditions.
Such cointegrals do not always exist.
For a multiplier Hopf algebra, if cointegrals exist, then they are unique (up to a scalar) and faithful (see \cite{VDZ}).
But this does not hold anymore in the weak multiplier Hopf algebra case.
Analogous to the integral, we assume that $A$ has a faithful set of cointegrals, i.e., $A$ is of {\it discrete type}.
Base on this assumption, we get some interesting results.

The paper is organized in the following way.

In Section 2, we recall some definitions which will be used in the following, such as weak multiplier Hopf algebras, left and right integrals,
and algebraic quantum groups. We also show some necessary properties.

In Section 3, we introduce the cointegrals in a regular multiplier Hopf algebra and give some properties of it.
First, we give some equivalent statements of a cointegral.
Next, We focus on the maps
$$
\omega\o h\mapsto (\omega\o id)\D(h)
$$
$$
\omega\o h\mapsto (id\o \omega)\D(h)
$$
from $A'\o H$ to $M(A)$. We show that these maps are injective on the balanced tensor product $A'\o_{A_s} H$.
And they are surjective if $A$ has a faithful set of cointegrals.
Finally, for a weak multiplier Hopf algebra $A$ of discrete type, we show that $A$ has integrals under certain conditions.
We also consider the uniqueness of the cointegrals.
When the cointegral space $H$ is the faithful, for any $h\in H$, we have $h\in HA_s$ and $h\in HA_t$, see Proposition \ref{unique}.

In Section 4, we consider a special type of discrete weak multiplier Hopf algebras.
That is a regular weak multiplier Hopf algebra with a single faithful cointegral.
And at this time we obtain some better results. As a consequence of Proposition \ref{unique}, we get more interesting formulas, see Proposition \ref{collection}. In this case, we prove that $A$ must have integrals,  and the set of integrals is  left faithful.
When the set is faithful, the dual algebra $\widehat{A}$ is a weak Hopf algebra. We also consider the Frobenius property of discrete weak multiplier Hopf algebras. At the end of this section, we consider the duality property between compact and discrete algebraic quantum groupoids.

In Section 5, we  study the cointegrals in some examples.

\section{Preliminaries}

In this paper, the ground field is the $\mathbb{C}$. We will work with the non-unital algebra $A$ but with non-degenerate product. The algebra $A$ is called idempotent if $A^2=A$. We will use $M(A)$ for the multiplier algebra of $A$. The multiplier algebra can be characterized as the largest algebra containing $A$ as a dense ideal. For any $a,b\in A, m\in M(A)$, by the definition we have $am\in A, mb\in A$  and $a(mb)=(am)b$. The identity element in the multiplier algebra will be denoted by $1$. If $A$ is non-degenerate and idempotent, then $A\o A$ is again a non-degenerate and idempotent algebra, so we can consider the multiplier algebra $M(A\o A)$.

We use $id$ for the identity map on any spaces and $\tau$ for the flip map. The space of all linear functionals on $A$ will be denoted by $A'$. A linear functional $f\in A'$ is called left faithful if $f(ba)=0$ for all $b$ implies $a=0$ and it is called right faithful if if $f(ba)=0$ for all $a$ implies $b=0$. $f$ is called faithful if $f$ is both left and right faithful.

Let $A$ and $B$ be two non-degenerate algebras, a homomorphism $f: A\ra M(B)$ is called weak non-degenerate if there exists an idempotent $E\in M(B)$ such that  $f(A)B=EB$ and $Bf(A)=BE$. If $f$ is weak non-degenerate, then $f$ has a unique extension to a homomorphism $M(A)\ra M(B)$. Usually we still denote the extended map by $f$. Note that if $E=1$, then the map is called non-degenerate.

A {\it coproduct} on $A$ is a homomorphism $\D: A\ra M(A\o A)$ such that

$\bullet$ the elements of the form
$$\D(a)(1\o b) \quad \text{and} \quad (a\o 1)\D(b)$$
belongs to $A\o A$ for all $a, b\in A$,

$\bullet$ $\D$ is coassociative in the sense that
$$(a\o 1\o 1)(\D\o id)(\D(b)(1\o c))=(id\o \D)((a\o 1)\D(b))(1\o 1\o c)$$
for all $a, b, c\in A$.

We can define $T_1: A\o A\ra A\o A$ by $T_1(a\o b)=\D(a)(1\o b)$ and $T_2: A\o A\ra A\o A$ by $T_2(a\o b)=(a\o 1)\D(b)$. The coproduct is called {\it regular} if $\D(a)(b\o 1)$ and $(1\o a)\D(b)$ are in $A\o A$. Now we can define the maps $T_3$ and $T_4$ on $A\o A$ by $T_3(a\o b)=(1\o b)\D(a)$ and $T_4(a\o b)=\D(b)(a\o 1)$. We will use $Ran(T_i)$ and $Ker(T_i)$ for the ranges and kernels of these maps.  $\D$ is called {\it full} if the smallest subspaces $V$ and $W$ of $A$ satisfying $$\D(A)(1\o A)\subset V\o A \quad \text{and}  \quad (A\o 1)\D(A)\subset A\o W$$
are both $A$ itself.

We will use leg numbering notion in this paper. For example, let $\D$ be a coproduct on $A$, then $\D_{13}$ means the map $\tau_{23}\circ(\D\o id)$ from $A$ to $M(A\o A\o A)$. Here $\tau_{23}$ equals to $id\o \tau$. When $E\in M(A\o A)$, we have $E_{12}=E\o 1$, $E_{23}=1\o E$ and $E_{13}=\tau_{23}(E_{12})$.

 The linear functional $\v$ on $A$ is called a {\it counit} if
$$(\v\o id)(\D(a)(1\o b))=ab \quad \text{and} \quad (id\o \v)((c\o 1)\D(a))=ca$$
for all $a, b, c\in A$.

\begin{definition}
A weak multiplier Hopf algebra is a pair $(A,\D)$ of a non-degenerate idempotent algebra with a full coproduct and a counit satisfying the following conditions.

(i) There exists an idempotent $E\in M(A\otimes A)$ giving the ranges of the canonical maps:
$$
T_1(A\otimes A)=E(A\otimes A)\quad {and} \quad T_2(A\otimes A)=(A\otimes A)E.
$$

(ii) The element $E$ satisfies
$$
(id \otimes \D)E=(E\otimes 1)(1\otimes E)=(1\otimes E)(E\otimes 1),
$$

(iii) The kernels of the canonical maps are of the form
$$
Ker(T_1)=(1-G_1)(A\otimes A),$$
$$
 Ker(T_2)=(1-G_2)(A\otimes A),
$$
where $G_1$ and $G_2$ are linear maps from $A\otimes A$ to itself, given as
$$
(G_1\otimes id)(\D_{13}(a)(1\otimes b\otimes c))=\D_{13}(a)(1\otimes E)(1\otimes b\otimes c),
$$
$$
(id\otimes G_2)((a\otimes b\otimes 1)\D_{13}(c))=(a\otimes b\otimes 1)(E\otimes 1)\D_{13}(c),
$$
for all $a, b, c\in A$.
\end{definition}

\begin{remark}
For the coproduct $\D$ on $A$, since we assume that $\D(a)(1\o b)$ and $(a\o 1)\D(b)$ are in $A\o A$. Then we can use Sweedler notation and write $\D(a)=\sum a_{(1)}\o a_{(2)}$. Sometimes we will omit the summation notation. We refer to \cite{VW3} and \cite{ZW} where the use of the Sweedler notation is explained in detail. In the following, when we meet new formulas or expressions, we will explain it.
\end{remark}

Let $A$ be a weak multiplier Hopf algebra, then there exists an antipode $S$ on $A$ such that
$$\sum a_{(1)}S(a_{(2)})a_{(3)}=a \quad \text{and} \quad \sum S(a_{(1)})a_{(2)}S(a_{(3)})=S(a).$$
We also know that $S$ is a non-degenerate anti-homomorphism and an anti-coalgebra map in the sense that $\D(S(a))=(\tau(S\o S)\D(a))E$. $A$ is called {\it regular} if $S$ is bijective. When $A$ is a regular weak multiplier Hopf algebra, we have $\D(S(a))=(\tau(S\o S)\D(a))$.

For any $a\in A$, we can define
$$\v_s(a)=\sum S(a_{(1)})a_{(2)} \quad \text{and} \quad \v_t(a)=\sum a_{(1)}S(a_{(2)}).$$
The map $\v_s$ is called the {\it source map} and $\v_t$ is called the {\it target map}. The images $\v_s(A)$ and $\v_t(A)$ are called {\it source algebra} and {\it target algebra}, respectively. In the regular case, we can define
$$\v'_s(a)=\sum a_{(2)}S^{-1}(a_{(1)}) \quad \text{and} \quad \v'_t(a)=\sum S^{-1}(a_{(2)})a_{(1)}.$$
Recall that there exists  distinguished linear functionals $\vp_B$ and $\vp_C$ on $\v_s(B)$ and $\v_t(C)$ respectively, defined by the formulas
$$(\vp_B\o id)E=1\quad  \text{and} \quad (id\o \vp_C)E=1.$$
For more information about  the distinguished linear functionals we refer to \cite{VD4}.

We will define $A_s=\{y\in M(A)|\D(y)=E(1\o y)\}$ and $A_t=\{x\in M(A)|\D(x)=(x\o 1)E\}$. Then $A_s$ and $A_t$ are commuting algebras of $M(A)$. We also know that $\v_s(A)$ is a right ideal of $A_s$ and $\v_t(A)$ is a left ideal of $A_t$. When $A$ is regular, the multiplier algebras $M(\v_s(A))$ and $M(\v_t(B))$ are respectively equal to $A_s$ and $A_t$.

For any $a\in A$ and any linear functional $f\in A'$, we can define elements $(id\o f)\D(a)$ and $(f\o id)\D(a)$ in $M(A)$. Denote $(id\o f)\D(a)$ by $x$, for any $b\in A$, the multiplier $x$ is defined as
$$xb=(id\o f)(\D(a)(b\o 1)) \quad \text{and} \quad bx=(id\o f)((b\o 1)\D(a)).$$
Similarly, we can define the multiplier $(f\o id)\D(a)$.

\begin{definition}
 A non-zero functional $\vp$ on $A$ is called a {\it left integral} if $(id\o \vp)\D(a)\in A_t$ for all $a\in A$.  Similarly, a {\it right integral} is a non-zero functional $\psi$ such that $(\psi\o id)\D(a)\in A_s$ for all $a\in A$.
\end{definition}

\begin{definition}
We say that $\int_L$ is a {\it left faithful set of integrals} if the condition hold: Given $x\in A$ we must have $x=0$ if $\vp(xa)=0$ for all left integrals $\vp\in \int_L$ and elements $a\in A$. $\int_L$ is called a {\it right faithful set of integrals} if $\vp(ax)=0$ for all left integrals $\vp$ and $a\in A$, then $x=0$. $\int_L$ is a {\it faithful set of integrals} if and only if $\int_L$ is both left and right faithful.
\end{definition}

\begin{definition}
Let $A$ be a regular weak multiplier Hopf algebra. If there exists a faithful set of integrals, then we call $A$ an {\it algebraic quantum groupoid}.
\end{definition}

Define $\widehat{A}=\{\vp(\cdot a)|a\in A, \vp \text{ is a left integral}\}$, it is a subspace of $A'$. Let $A$ be an algebraic quantum groupoid, then the space $\widehat{A}$ is again an algebraic quantum groupoid with the dual structure. So we will call it the {\it dual} of $A$.

\section{Weak multiplier Hopf algebras of discrete type}

Let $(A, \D, \v, S, E)$ be a  weak multiplier Hopf algebra.
\begin{definition}
 A non-zero element  $h\in M(A)$ is called a {\it left cointegral} if it satisfies
$ah=\v_t(a)h $
for all $a\in A$. Similarly, a  non-zero element  $h\in M(A)$ is called a {\it right cointegral} if it satisfies
$ka=k\v_s(a)$
for all $k\in A$.
\end{definition}

Before we continue, we would like to make a few remarks.

\begin{remark}
i) First observe that we allow cointegrals to sit in the multiplier algebra. We do not require them to be the elements of $A$ itself. In the case of multiplier Hopf algebras, where we have $\v_t(a)=\v(a)1$ and $\v_s(a)=\v(a)1$ for all $a$, the defining requirements imply that the cointegrals are in $A$. This is not the case here. Of course, if the images of the source and target maps contain the identity, then this will also imply that cointegrals  must lie in $A$ itself. But this is not always the case.

ii) If $h$ is a left cointegral in $M(A)$, we will have that $ha$ is also a left cointegral for all $a$, and now $ha\in A$. If $ha=0$ for all $a$, then $h=0$. So if there is  a cointegral in $M(A)$ then there is also one in $A$. Therefore, it it in general enough to consider only cointegrals in the algebra itself.

iii) In this case, cointegrals need not be unique.
\end{remark}

\begin{example}\label{example1}
Let $G$ be a groupoid. Consider the weak multiplier Hopf algebra $A$ of complex functions with finite support and pointwise operations. Any $h$ with  support in the units will be a cointegral. In fact, for any $g\in A, p\in G$, we have
$$
(gh)(p)=g(p)h(p)$$
$$
(\v_t(g)h)(p)=g(pp^{-1})h(p).
$$
If $p$ is a unit, then $g(p)=g(pp^{-1})$. So $h$ is a left cointegral. If the set of units is infinite, then $h$ is in $M(A)$.
\end{example}

We now prove the first main properties of left cointegrals.

 \begin{proposition}\label{1}
Let $h\in M(A)$, for any $a\in A$. The following statements are equivalent:

$(1)$ $ah=\v_t(a)h$ ,

$(2)$ $(1\o a)\D(h)=(S(a)\o 1)\D(h)$ ,

$(3)$ $S(h)a=S(h)\v_s(a)$ ,

$(4)$ $\D(a)(h\o 1)=E(h\o a)$ ,

$(5)$ $\D(a)(1\o h)=E(a\o h)$ ,

$(6)$ $S(a)h=\v_s(a)h$ .

\end{proposition}
\begin{proof}
(1)$\Rightarrow$ (2):
\begin{align*}
(1\o a)\D(h) = (1\o a)E\D(a) &= (S(a_{(1)})\o 1)\D(a_{(2)}h)\\
                             &= (S(a_{(1)})\o 1)\D(\v_t(a_{(2)})h)\\
                             &= (S(a_{(1)})\o 1)(\v_t(a_{(2)})\o 1)E\D(h)\\
                             &= (S(a)\o 1)\D(h)
\end{align*}

(2)$\Rightarrow$ (1): Let $b\in A$, we have
$$bah=ba_{(1)}h_{(1)}\v(a_{(2)}h_{(2)})=ba_{(1)}S(a_{(2)})h_{(1)}\v{(h_{(2)})}=b\v_t(a)h$$
here $b$ is used for the covering. By the non-degeneracy of the product, we get $ah=\v_t(a)h$.

(1)$\Rightarrow$ (3): Applying $S$ to $ah=\v_t(a)h $, then we can get $S(h)S(a)=S(h)\v_s(S(a))$, so $S(h)$ is a right cointegral. Conversely, from $(3)$ to $(1)$ we only need to apply $S^{-1}$.

(2)$\Rightarrow$ (4): For any $y\in A_s$, we have
$$
\D(a)(yh)=(1\o y)\D(h)=(S(y)\o 1)\D(h).
$$
Applying $id\o \v$ we can get $yh=S(y)h$. Next we have
\begin{align*}
\D(a)(h\o 1) &= \v_t(a_{(1)})h\o a_{(2)}\\
             &= a_{(1)}S(a_{(2)})h\o a_{(3)}\\
             &= (S\o id)E(h\o a).
\end{align*}
 Because $yh=S(y)h$ and the left leg of $E$ is $\v_s(A)$, so we can get
 $$
 \D(a)(h\o 1)=E(h\o a).
 $$

(4)$\Rightarrow$ (1): If we have $\D(a)(h\o 1)=E(h\o a)$ for all $a$, then
\begin{align*}
S^{-1}(a_{(2)})a_{(1)}h &= S^{-1}(a)S^{-1}(E_{(2)})E_{(1)}h\\
                        &= S^{-1}(a)h.
\end{align*}
Replacing $a$ by $S(a)$ then we can get $ah=a_{(1)}S(a_{(2)})h=\v_t(a)h$.

(1)$\Rightarrow$ (5):
\begin{align*}
\D(a)(1\o h)&=a_{(1)}\o a_{(2)}h\\
            &= a_{(1)}\o a_{(2)}S(a_{(3)})h\\
            &=T_1R_1(a\o h)=E(a\o h)
\end{align*}

(5)$\Rightarrow$ (1): Applying $\v\o id$ to the formula $\D(a)(1\o h)=E(a\o h)$ then we can prove the result.

(1)$\Rightarrow$ (6): For any $a\in A$, we have
$$
S(a)h = \v_t(S(a))h
      = S(\v_s(a))h
      = \v_s(a)h.
$$
In the last equality we use the formula $yh=S(y)h$.

(6)$\Rightarrow$ (1):
$$ah=a_{(1)}S(a_{(2)})a_{(3)}h=a_{(1)}\v_s(a_{(2)})h \overset{(6)}{=}a_{(1)}S(a_{(2)})h=\v_t(a)h.$$
If we multiply with an element in $A$ from the left, then we can get the necessary coverings.
\end{proof}

From the proof above we have the following useful formula.

\begin{corollary}
If $h$ is a left cointegral, then $yh=S(y)h$ for all $y\in A_s$.
\end{corollary}

From this corollary  it follows that $A_sh=A_th$ for any left cointegral $h$. We also know that $h$ is a left cointegral if and only if $S(h)$ is a right cointegral. For right cointegrals, we have similar results.

\begin{proposition}
Let $k\in M(A)$, for all $a\in A$,  the following statements are equivalent:

$(1)$ $ka=k\v_s(a)$,

$(2)$ $(1\o k)\D(a)=(a\o k)E$,

$(3)$ $(k\o 1)\D(a)=(k\o a)E$,

$(4)$ $\D(k)(a\o 1)=\D(k)(1\o S(a))$,

$(5)$ $kS(a)=k\v_t(a)$.
\end{proposition}
And for the right cointegral $k$, we have $kx=kS(x)$ whenever $x\in A_t$.

For any left left cointegral $h$, we can define
$$\gamma(h)=\sum h_{(1)}S(\v_s(h_{(2)})).$$
It is meaningful if we multiply with an element of $A$ from the left. And we have the following result.
\begin{proposition}\label{gamma}
$\g(h)$ is a left cointegral.
\end{proposition}
\begin{proof}
For any $a\in A$, let us compute
\begin{align*}
S(a)\g(h) &= S(a)h_{(1)}S(\v_s(h_{(2)}))\\
          &= h_{(1)}S(S(a_{(1)}h_{(2)})a_{(2)}h_{(3)})\\
          &\overset{ y=\v_s(a)}{=} h_{(1)}S(S(h_{(2)})yh_{(3)})\\
          &= h_{(1)}S(S(y_{(1)}h_{(2)})y_{(2)}h_{(3)})\\
          &= S(y)h_{(1)}S(S(h_{(2)})h_{(3)})\\
          &= \v_t(S(a))\g(h).
\end{align*}
In the second and fifth equalities we use (2) of Proposition \ref{1}.
 In the fourth equality we use $\v_s\circ \v_s=\v_s$. For the last equality we use $S\circ \v_s=\v_t\circ S$. For the third equality, we can multiply with $a$ from the left. Then the expression becomes $ ah_{(1)}S(S(h_{(2)})yh_{(3)}e)$, $e$ is a local unit. Now we have everything covered here.
\end{proof}

In what follows, assume that $A$ is a regular WMHA and we will denote by $H$ the space of left cointegrals in $A$. We study the maps
$$
\omega\o h\mapsto (\omega\o id)\D(h)
$$
$$
\omega\o h\mapsto (id\o \omega)\D(h)
$$
from $A'\o H$ to $M(A)$. We use regularity of the coproduct to have that the image of these two maps indeed are in $M(A)$.

In the case of  multiplier Hopf algebras, we know that these maps are bijective, see Proposition 2.8 in \cite{VDZ}. This will no longer be the case here. However, we can characterize the kernels of these maps. This is essentially obtained in the following proposition.

\begin{proposition}\label{injective}
Define a right action of $A_s$ on $A'$ by $\omega\leftharpoonup y=\omega(\cdot S(y))$ and a left action of $A_s$ on $H$ by $y\rightharpoonup h=hS(y)$. Then the map $\omega\o h\mapsto (\omega\o id)\D(h)$ is well-defined and injective on the balanced tensor product $A'\o_{A_s} H$.
\end{proposition}
\begin{proof}
It is easy to check that they are actions. Next we will show that the map is well-defined. Since $S(y)\in A_t$, we have
$$(\omega\o id)(\D(hS(y)))=(\omega\o id)(\D(y))(S(y)\o 1)=(\omega(\cdot S(y))\o id)\D(h).$$
Now let us prove the injectivity of the map. Assume that $\sum (\omega_i\o id)\D(h_i)=0$. Then
$$
\sum (\omega_i\o id)(1\o a)\D(h_i)=0 \quad\quad\text{for all $a$,}
$$
by Proposition \ref{1} we can get
$$
\sum (\omega_i\o id)(S(a)\o 1)\D(h_i)=0 \quad\quad\text{for all $a$,}
$$
this is equivalent with
$$
\sum (\omega_i\o id)(a\o 1)\D(h_i)=0 \quad\quad\text{for all $a$.}
$$
Applying $\D$ and $S^{-1}$, then
$$
\sum \omega_i(abh_{i(1)})S^{-1}(h_{i(2)})\o h_{i(3)}=0 \quad\quad\text{for all $a, b$.}
$$
Note that if we multiply with an element from the right, then $h_{i(3)}$ is covered. Now we have
$$
\sum \omega_i(a\cdot h_{i(1)})\o S^{-1}(h_{i(2)})\o h_{i(3)}=0 \quad\quad\text{for all $a$.}
$$
So
$$
\sum \omega_i(aS^{-1}(h_{i(2)})h_{i(1)})h_{i(3)}=0 \quad\quad\text{for all $a$.}
$$
Since $S^{-1}(h_{i(2)})h_{i(1)}\in \v_t(A)$, so
$$
\sum (\omega_i\o id)(a\o h_{i(3)}S^{-1}(h_{i(2)})h_{i(1)})=0 \quad\quad\text{for all $a$.}
$$
Finally we can obtain $\sum \omega_i\o h_i=0$.
\end{proof}

\begin{remark}
If we let $A_s$ act on $H$ from the right by multiplication and on $A'$ from the left by $y\rightharpoonup \omega=\omega(\cdot y)$. Then the map $\omega\o h\mapsto (\omega\o id)\D(h)$ is injective on balanced tensor product $A'\o_{A_t} H$. The result for the other map is similar, but with $A_t$ instead of $A_s$.
And also similar results should be true for right cointegrals. We leave it to the reader to formulate these results and to prove them.
\end{remark}

Now, we turn our attention to the problem of surjectivity of these maps.

In the paper \cite{VW5}, Van Daele and Wang first defined the notion of the left leg and the right leg of an element $\D(a)$ for a fixed element $a$ or for $a$ is a subspace of $A$. The notion is also appeared  in the definition of weak multiplier Hopf algebra. Let us recall this notion again and apply it immediately for the space $H$ of left cointegrals in $A$.

\begin{definition}
By the left leg of $\D(H)$ we mean the smallest subspace $V$ of $A$ satisfying $\D(h)(1\o a)\in V\o A$ for all $a\in A$ and $h\in H$. Similarly, the right leg of $\D(H)$ is the smallest subspace $W$ of $A$ satisfying $(a\o 1)\D(h)\in A\o W$ for all $a\in A$ and $h\in H$.
\end{definition}

\begin{proposition}\label{faithful}
The following are equivalent:

$1)$ The right leg of $\D(H)$ is all of $A$,

$2)$ Given $y\in A_s$, then $y=0$ if $hy=0$ for all $h\in H$.

Also the following are equivalent:

$3)$ The left leg of $\D(H)$ is all of $A$,

$4)$ Given $x\in¡¡A_t$, then $x=0$ if $hx=0$ for all $h\in H$.
\end{proposition}
\begin{proof}
1)$\Rightarrow$ 2):  If $hy=0$ for all $h\in H$, then $\D(hy)=0$. By $\D(y)=E(1\o y)$ we have
$$
\D(h)(1\o y)=0
$$
for all $h$. For any $\omega\in \widehat{A}$, then
$$(\omega\o id)\D(h)y=0.$$
Since the right leg is all of $A$, now we can obtain $y=0$.

2)$\Rightarrow$ 1): Suppose that the right leg of $\D(H)$ is not $A$. Then there is a $\rho\in A'$ so that
$$\rho((\omega\o id)((a\o 1)\D(h)))=0$$
for all $\omega\in \widehat{A}, a\in A$ but $\rho\neq 0$. We can rewrite this as
$$\omega((id\o \rho)((a\o 1)\D(h))))=0$$
for all $\omega, a$. Therefore
$$
(id \o \rho)\D(h)=0
$$
for all $h\in H$. Next we will show that $\rho$ must be $0$, then we finish the proof. Using the formula in Proposition \ref{1}, we obtain
\begin{align*}
a(id\o \rho)\D(h) &= (id\o \rho)((a\o 1)\D(h))\\
                  &= (id\o \rho)((1\o S^{-1}(a))\D(h))\\
                  &=0
\end{align*}
for all $a, h$. Then
$$(id\o \rho)((1\o a)\D(h))=0$$
for all $a, h$. Applying $\D$ and $S$ we have
$$\sum h_{(1)}\rho(aS(h_{(2)}h_{(3)}))=0.$$
Since we have $\sum h_{(1)}\o aS(h_{(2)})h_{(3)}=(h\o a)(id\o S)E$, so
$$
(id\o \rho)(h\o a)(id\o S)E=0
$$
for all $a, h$. Now by the assumption we can get
$$
(id\o \rho)(1\o a)(id\o S)E=0
$$
for all $a$. Applying the distinguished linear functional $\vp_B$, it implies $\rho(a)=0$ for all $a$. So $\rho=0$.

The proof of  (3)$\Leftrightarrow$(4) is similar.
\end{proof}

If the left leg of $\D(H)$ is all of $A$, then we have
\begin{align*}
A &=span\{(id\o \omega(\cdot a))\D(h)|\forall \omega\in A', a\in A, h\in H\}\\
  &=span\{(id\o \omega(a \cdot))\D(h)|\forall \omega\in A', a\in A, h\in H\}.
\end{align*}
Similarly on the other side. Now we can come to our main definition.

\begin{definition}
We say that $H$ is a faithful set of left cointegrals if the items (1) and (3) in the above proposition are satisfied. We call $(A, \D)$ of {\it discrete type} if $A$ has  a faithful set of left cointegrals.
\end{definition}

Let us also recall the definition of faithful cointegral here.

\begin{definition}
Let $A$ be a regular weak multiplier Hopf algebra with a cointegral $h$. $h$ is called faithful if the following assertion holds: For any $f\in A'$, if $(id\o f)\D(h)=0$ or $(f\o id)\D(h)=0$, then $f$ must be $0$.
\end{definition}

From the proof of  Proposition \ref{faithful} we know that if there is only one cointegral, then the conditions (1) and (3) are the same with the faithfulness of the  cointegral.   We can also compare this notion with the notion of faithful set of integrals which is defined in Definition 1.11 in  \cite{VW4}.
Obviously, a regular weak multiplier Hopf algebra with a single faithful cointegral is of discrete type. And we will turn our attention towards this special type in the next section.

Denote $\widetilde{A}=\{\omega(\cdot a)|a\in A, \omega\in A'\}$.

\begin{proposition}\label{bijective-H}
Assume that $A$ is a  weak multiplier Hopf algebra of discrete type,  $H$ is the faithful set of cointegrals. Then the two maps
 $$f_s: \widetilde{A}\o_{A_s} H\ra A,\quad \omega\o h \mapsto (\omega\o id)\D(h)$$
 $$f_t: \widetilde{A}\o_{A_t} H\ra A,\quad \omega\o h \mapsto (id\o \omega)\D(h)$$
are bijective.
\end{proposition}
\begin{proof}
The injectivity of these maps is from Proposition \ref{injective}. Since $H$ is faithful, the ranges are all of $A$.
\end{proof}

Moreover, we can extend these maps to all of $A'$. Note that for any $\omega=f(\cdot a)\in \widetilde{A}, x\in M(A)$, we can define $\omega(x)=f(xa)$. Assume that $\omega=f(\cdot a)=g(\cdot b)$, then we have to show that $f(xa)=g(xb)$. This is true because we have local units for $A$. We can find $e$ such that $ea=a$ and $eb=b$, then $f(xa)=f((xe)a)=g((xe)b)=g(xb)$. So $\omega(x)$ is well-defined.  The extension is still denoted by $\omega$.

\begin{proposition}\label{multiplier}
Let $H$ be the faithful set of cointegrals of $A$. Then the maps
$$f_s: A'\o_{A_s} H\ra M(A),\quad \omega\o h \mapsto (\omega\o id)\D(h)$$
 $$f_t: A'\o_{A_t} H\ra M(A),\quad \omega\o h \mapsto (id\o \omega)\D(h)$$
are bijective.
\end{proposition}
\begin{proof}
Note for these extended maps we still use the same symbols. We only need to show that they are surjective. For any $x\in M(A), \omega\in \widetilde{A}, h\in H$, define $f\in A'$ by
$$f((\omega\o id)\D(h))=\omega(x).$$
This is well-defined because every element in $A$ can be uniquely written as $(\omega\o id)\D(h)$. Then we have $\omega((f\o id)\D(h))=\omega(x)$. Since $\widetilde{A}$ separates points in $A$, so $x=(f\o id)\D(h)$.

The proof of the second map is similar.
\end{proof}

\begin{proposition}
Let $A$ be a weak multiplier Hopf algebra of discrete type. Then there exists a bijective map between  $M(A)$ and the dual space $(\widetilde{A})'$ of $\widetilde{A}$.
\end{proposition}
\begin{proof}
For any $x\in M(A), \omega\in (\widetilde{A})'$, define
$$f: M(A)\ra (\widetilde{A})',\quad  f(x)(\omega)=\omega(x).$$
Note that $\omega(x)$ is meaningful, so $f$ is well-defined.

If $f(x)=0$ then $\omega(x)=0$ for all $\omega$. Any  $\omega$ has the form $\vp(\cdot a)$, so we have $\vp(xa)=0$ for any $\vp\in A', a\in A$. Then we have $x=0$. So $f$ is injective.

Note that by Proposition \ref{bijective-H} for any $a\in A$, we can write $a=(id\o \omega)\D(h)$ where $\omega\in \widetilde{A}$.  For any $\Gamma\in (\widetilde{A})'$, we can define $\psi\in A'$ by $\psi(a)=\Gamma(\omega)$. Now let $x=(\psi\o id)\D(h)\in M(A)$ then we have $f(x)(\omega)=\Gamma(\omega)$. So $f(x)=\Gamma$. It means that $f$ is surjective.
\end{proof}

\begin{proposition}\label{integral}
Let $A$ be a weak multiplier Hopf algebra of discrete type, $H$ is the faithful set of cointegrals. Let $\vp\in A'$,  then $\vp$ is a left integral if and only if $(id\o \vp)\D(h)\in A_t$ for all $h\in H$.
\end{proposition}
\begin{proof}
Assume that $(id\o \vp)\D(h)\in A_t$ for all $h\in H$. Note that  we have Proposition \ref{multiplier}, so $(id\o \vp)\D(h)$ in $A_t$ is possible.  For any $a\in A$, we can write $a=(f\o id)\D(h')$ where $f=\psi(\cdot b)$, $\psi\in A', b\in A, h'\in H$.
 For any $c\in A$, we have
\begin{eqnarray*}
c(id\o \vp)\D(a) &=& ca_{(1)}\vp(a_{(2)})\\
                 &=&\psi(h'_{(1)}b)ch'_{(2)}\vp(h'_{(3)})\\
                 &=& \psi(E_1h'_{(1)}\vp(h'_{(2)}))cE_2.
\end{eqnarray*}
Here $c$ is used for the coverings. So we can get $(id\o \vp)\D(a)\in A_t$ and this means that $\vp$ is a left integral. Conversely, obviously.
\end{proof}

Now let us consider the uniqueness of the cointegrals. And we have the following result.
\begin{proposition}\label{unique}
Let $H$ be the set of all left cointegrals. If the right leg of $\D(H)$ is all of $A$, then we have $h'\in HA_t$ and $h'\in HA_s$ for any $h'\in H$.
\end{proposition}
\begin{proof}
For any $h\in H$, now $S(h')$ is a right cointegral. We have
\begin{align*}
(1\o h')\D(h)&= (S(h')\o 1)\D(h)\\
             &= (S(h')\o h)E.
\end{align*}
Applying any $\omega\in \widehat{A}$ to the formula, then we have
$$
h'(\omega\o id)\D(h)=h(\omega\o id)(S(h')\o 1)E.
$$
Since the right leg of $\D(H)$ is $A$ we find that $h'A\subseteq HA_t$. We also have
\begin{align*}
(h'\o 1)\D(h) &= (1\o S^{-1}(h'))\D(h)\\
              &= (h\o S^{-1}(h'))E,
\end{align*}
so we can get $h'A\subseteq HA_s$. Note that we have local units, then we find
$$h'\in HA_t \quad \quad \text{and} \quad \quad h'\in HA_s.$$
\end{proof}

For the right cointegrals we have similar results.

Now we have $h'\in HA_t$ and $h'\in HA_s$, then it will follow that there exists a map $\rho: A_s\ra A_t$ such that
$$h'y=h\rho (y)$$
for some $y\in A_s$.  We will find more results in the case of a single faithful cointegral. It will be discussed in the next section.

\begin{proposition}
Let $h, h'$ be two left cointegrals, then $(h'\o 1)\D(h)=(\g(h)\o S^{-1}(h'))E$.
\end{proposition}
\begin{proof}
Note that $\g$ is defined in Proposition \ref{gamma}.
\begin{eqnarray*}
(h'\o 1)\D(h) &=& (1\o S^{-1}(h'))\D(h)\\
              &=& (h_{(1)}\o S^{-1}(h')h_{(2)})E\\
              &=& (h_{(1)}\o S^{-1}(h')\v_s(h_{(2)}))E\\
              &=& (h_{(1)}S(\v_s(h_{(2)}))\o S^{-1}(h'))E\\
              &=& (\g(h)\o S^{-1}(h'))E
\end{eqnarray*}
In the fourth equality we use $(1\o y)E=(S(y)\o 1)E$ for any $y\in A_s$.
\end{proof}

\section{The case of a single faithful  cointegral}

Let $A$ be a regular weak multiplier Hopf algebra. In this section we will assume that there exists a single faithful cointegral in $A$. Then we will find some better results following the previous section. The first result is a consequence of Proposition \ref{faithful}. Note that this condition is much stronger, however for multiplier Hopf algebras it is possible.

\begin{proposition}\label{first}
Assume that $h$ is the faithful cointegral in $A$. Then the maps
$$ \omega\ra (id\o \omega)\D(h) \quad \text{and} \quad  \omega\ra (\omega\o id)\D(h)$$
are bijective from $\widetilde{A}$ to $A$.
\end{proposition}
\begin{proof}
Since $h$ is faithful, these maps are injective. Now we will show that they are surjective. Suppose that the first map is not surjective. Then there is a $\varphi\in A'$ such that $\vp((id\o \omega)\D(h))=0$ but $\vp\neq 0$. This is equivalent with $\omega((\vp\o id)\D(h))=0$ for all $\omega$. So $(\vp\o id)\D(h)=0$. Again by the faithfulness of $h$, we can get $\vp=0$, contradiction. So the first map is surjective. Similarly for the second map.
\end{proof}

As a corollary we can get  the relation between left cointegral $h$ and right cointegral $S(h)$.
\begin{corollary}
Assume that $h$ is the faithful cointegral. Then there is an element $\d$ in $A'$ so that $S(h)=(id\o \d)\D(h)$.
\end{corollary}

Let $h$ be the single faithful left cointegral in $A$. From Proposition \ref{unique} we have $ha\in hA_t$ and $ha\in hA_s$ for any $a\in A$. Then we can define
$$
\g_t: A\ra A_t\quad \text{and}\quad
\g_s: A\ra A_s
$$
by
$$
ha=h\g_t(a)\quad \text{and}\quad
ha=h\g_s(a),
$$
respectively. Since $h$ is faithful, $\g_t(a)$ and $\g_s(a)$ are completely determined.

\begin{proposition}
For any $a\in A, x\in A_t, y\in A_s$, we have $\g_t(ax)=\g_t(a)x$ and $\g_s(ay)=\g_s(a)y$.
\end{proposition}
\begin{proof}
We have
$$hax=h\g_t(a)x=h\g_t(ax).$$
Since $h$ is faithful we have $\g_t(ax)=\g_t(a)x$. Similarly for the other one.
\end{proof}

We can also define a map $S_1: A_s\ra A_t$ by $hy=hS_1(y)$ for any $y\in A_s$. Indeed $hy\in hA_t$ and so $S_1(y)$ is determined.

\begin{proposition}
For any $y,y'\in A_s$, $S_1(yy')=S_1(y')S_1(y).$
\end{proposition}
\begin{proof}
\begin{align*}
h(yy')&= (hy)y'\\
      &= hS_1(y)y'\\
      &= hy'S_1(y)\\
      &= hS_1(y')S_1(y).
\end{align*}
In the third equality we use the fact that $A_s$ and $A_t$ are commuting subalgebras.  So $S_1(yy')=S_1(y')S_1(y)$
\end{proof}

\begin{proposition}\label{formula1}
For any $y\in A_s$, we have $\D(h)(1\o y)=\D(h)(S_1(y)\o 1)$.
\end{proposition}
\begin{proof}We have
\begin{align*}
\D(h)(1\o y) &= \D(h)E(1\o y)\\
             &= \D(h)\D(y)\\
             &= \D(h)\D(S_1(y))\\
             &= \D(h)(S_1(y)\o 1).
\end{align*}
\end{proof}

If we consider the map $S_2: A_t\ra A_s$ which is defined by $hx=hS_2(x), x\in A_t$.  Then we have the following result.
\begin{proposition}\label{formula2}
$S_2$ is an anti-homomorphism from $A_t$ to $A_s$. We also have
$$\D(h)(x\o 1)=\D(h)(1\o S_2(x))$$
for any $x\in A_t$.
\end{proposition}
The proof is similar.

Inspired by the formulas in Proposition \ref{formula1} and \ref{formula2}, we will get more formulas as a completion of Proposition \ref{1}.

For any $a\in A$, we have
\begin{align*}
\D(h)(a\o 1) &= \D(h)E(a\o 1)\\
             &= \D(ha_{(1)})((1\o S(a_{(2)})))\\
             &= \D(h\g_s(a_{(1)}))(1\o S(a_{(2)}))
\end{align*}
In the second equality we use $E(a\o 1)=\sum \D(a_{(1)})(1\o S(a_{(2)}))$. In the third equality we use $\D(y)=E(1\o y)$. And $a_{(1)}$ is covered by $h$.
 Define $S'_2(a)=\sum \g_s(a_{(1)})S(a_{(2)})$, then
 $$
 \D(h)(a\o 1)=\D(h)(1\o S'_2(a)).
 $$

We can also start with $\D(h)(1\o a)$. Now we need to define the map $S'_1(a)=\sum \g_t(a_{(2)})S^{-1}(a_{(1)})$, then  we can get
$$
\D(h)(1\o a)=\D(h)(S'_1(a)\o 1).
$$

Finally let us consider the expressions $(1\o h)\D(a)$  and $(h\o 1)\D(a)$. If we denote $(a\o 1)E_t=(id\o \g_t)\D(a)$ and $(1\o a)E_s=(\g_s\o id)\D(a)$, then $$
(1\o h)\D(a)=(a\o h)E_t \quad \text{and} \quad (h\o 1)\D(a)=(h\o a)E_s.
$$
The proof is easy. Let us check the first formula. We have
$$(1\o h)\D(a)=a_{(1)}\o h\g_t(a_{(2)})=(a\o h)E_t.$$

Note that $S'_1(a)$ and $S'_2(a)$ are well-defined in $M(A)$. And we   have $hS'_1(a)=h\v'_s(a)$ and $hS'_2(a)=h\v_s(a)$.

 Now let us collect all the formulas in the following proposition.
\begin{proposition}\label{collection}
With the notions above, for any $a\in A$, we have

$(1)\, \D(a)(1\o h)=E(a\o h),$

$(2)\, (1\o h )\D(a)=(a\o h)E_t,$

$(3)\, \D(a)(1\o h)=E(h\o a),$

$(4)\, (h\o 1)\D(a)=(h\o a)E_s,$

$(5)\, (1\o a)\D(h)=(S(a)\o 1)\D(h),$

$(6)\, \D(h)(1\o a)=\D(h)(S'_1(a)\o 1),$

$(7)\, (a\o 1)\D(h)=(1\o S^{-1}(a))\D(h),$

 $(8)\, \D(h)(a\o 1)=\D(h)(1\o S'_2(a)).$
\end{proposition}

Note that formulas (5) and (7) are the same formula.

Our next goal is to show that if $A$ has a faithful cointegral, then $A$ must has  integrals. Similar with the Proposition \ref{multiplier}, first we will extend the two bijective maps in Proposition \ref{first} to all of $A'$.

\begin{lemma}\label{bijective}
Let $h$ be the faithful left cointegral. Define
$$F_1: f\ra (id\o f)\D(h) \quad \text{and} \quad F_2: f\ra (f\o id)\D(h).$$
Then $F_1$ and $F_2$ are bijective maps  from $A'$ to $M(A)$.
\end{lemma}
\begin{proof}
The proof is similar with Proposition \ref{multiplier}.
\end{proof}

\begin{theorem}\label{theorem}
If $A$ is a regular weak multiplier Hopf algebra with a single faithful left  cointegral $h$, then $A$ has   integrals.
\end{theorem}
\begin{proof}
Since the map $F_1$ is bijective, so we can choose $\vp\in A'$ such that $(id\o \vp)\D(h)\in A_t$. By Proposition \ref{integral}, we know that every $\vp$ with the property is a left integral.

Moreover, we can choose $\psi\in A'$ such that $(\psi\o id)\D(h)\in A_s$. This is possible by the bijectivity of $F_2$. Then we can obtain  right integrals.
\end{proof}

Let $A$ be a regular weak multiplier Hopf algebra with a faithful cointegral $h$, the space of corresponding left (right) integrals on $A$  will be denoted by $\int_L(\int_R)$. One may expect that $\int_L(\int_R)$ is still faithful. Note that if $A$ is multiplier Hopf algebra, the (co)integral is automatically faithful when the (co)integral exists. In general this is not true for weak multiplier Hopf algebras.

\begin{proposition}
 $\int_L$ is right faithful. $\int_L$ is left faithful if and only if $S'_1$ is injective.
\end{proposition}
\begin{proof}
Assume that $a\in A$ and $\vp(ba)=0$ for all $b\in A, \vp\in \int_L$. By the Proposition \ref{first} we can write $b$ as $(\omega\o id)\D(h)$ with $\omega\in \widetilde{A}$. Then we have
$$(\omega\o \vp)(\D(h)(1\o a))=0$$
for any $\omega\in \widetilde{A}, \vp\in \int_L$. Using the formula (6) in Proposition \ref{collection}, we get
$$(\omega\o \vp)(\D(h)(S'_1(a)\o 1))=0$$
for any $\omega\in \widetilde{A}, \vp\in \int_L$. We can find $\vp_h\in \int_L$ such that $(id\o \vp_h)\D(h)=1$, then  we have $\omega(S'_1(a))=0$ for any $\omega\in \widetilde{A}$. Then $S'_1(a)=0$. So $\int_L$ is left faithful if and only if $S'_1$ is injective.

On the other hand, assume that $\vp(ab)=0$ for all $b$ and $\vp$. Still write $b$ as $(\omega\o id)\D(h)$. Now we use the formula (5) in Proposition \ref{collection} and choose $\vp_h$, then we get $\omega(S(a))=0$. So  we can  get $a=0$. This proves that $\int_L$ is right faithful.
\end{proof}

For the space $\int_R$ of  right integrals, we have the following result.

\begin{proposition}
$\int_R$ is right faithful. $\int_R$ is left faithful if and only if $S'_2$ is injective.
\end{proposition}

The proof is similar.

As we see in the above proof, we have a  left integral $\vp_h$ and a right integral $\psi_h$ such that
$(id\o \vp_h)\D(h)=1$ and $(\psi_h\o id)\D(h)=1$. Now we can show that the counit is in $\widetilde{A}$.

\begin{proposition}\label{counit}
For any $a\in A$, we have $\vp_h(ah)=\v(a)=\psi_h(ha)$.
\end{proposition}
\begin{proof}
Recall  $E(1\o a)=\v'_s(a_{(1)})\o a_{(2)}$ and $(id\o \vp_h)\D(h)=1$, we have
$$\vp_h(ah)=\vp_h(\v_t(a)h)=\vp_h(E_2h)\v(E_1a)=\vp_h(h_{(2)})\v(\v'_s(h_{(1)})a)=\v(a).$$
If we use $(a\o 1)E=a_{(1)}\o \v'_t(a_{(2)})$ then we can obtain another equality.
\end{proof}

 Next we will assume that the space $\int_L$ is faithful, then we can define the dual weak multiplier Hopf algbera $\widehat{A}$.

\begin{proposition}\label{unit}
Let $A$ be a regular weak multiplier Hopf algebra with a faithful left cointegral $h$. If the related integral space $\int_L$ is faithful then $\widetilde{A}=\widehat{A}$, moreover we have $\v\in \widehat{A}$, then $\widehat{A}$ is a weak Hopf algebra.
\end{proposition}
\begin{proof}
We only need to show that $\widetilde{A}\subset \widehat{A}$.  Take the left integral $\vp_h$, for any $a\in A$, we have
\begin{eqnarray*}
(id\o \vp_h)((1\o a)\D(h))&=& (id\o \vp_h)((S(a)\o 1)\D(h))\\
                          &=& S(a).
\end{eqnarray*}
By Proposition \ref{first}, there exists a $f\in \widetilde{A}$ such that
$$(id\o f)\D(h)=S(a)=(id\o \vp_h)((1\o a)\D(h)).$$
By the faithfulness of $h$ we get $f=\vp_h(a\cdot)$. So $\widetilde{A}\subset \widehat{A}$.  Then $\widetilde{A}= \widehat{A}$.

By the Proposition 4.12 in \cite{VW2}, we know that if $A$ has an identity, then $A$ is a weak Hopf algebra. Note that in Proposition \ref{counit} we have shown that $\v\in \widehat{A}$, so $\widehat{A}$ is a weak Hopf algebra.

\end{proof}

Now we will investigate when the regular weak multiplier Hopf algebra $A$ has  a cointegral.  First let us recall some   notions from \cite{VDZ}. We can consider the dual space $A'$ as a left $A$-module for the action defined by $a\omega=\omega(\cdot a)$ where $a\in A, \omega\in A'$. Note that $\widetilde{A}=AA'$ and the left $A$-module $\widetilde{A}$ is unital.

\begin{definition}
Let $A$ be an algebra (possibly infinite dimensional, possibly without an identity). If there is an $A$-module isomorphism from $A$ to $AA'$, then we call $A$ a {\it Frobenius} algebra. Let $I$ be a left ideal of $A$, the right annihilator of $I$ will be denoted by $r(I)$, the left annihilator  of $I$ will be denoted by $l(I)$. An algebra $A$ is called {\it quasi-Frobenius} if for any left ideal $I$ and any right ideal $J$ we have $lr(I)=I$ and $rl(J)=J$.
\end{definition}

\begin{proposition}
If $A$ has a  faithful cointgeral $h$, then $A$ is Frobenius.
\end{proposition}
\begin{proof}
By the Proposition \ref{first} we know that the map
$$F: \widetilde{A}\ra A, f\mapsto (id\o S(f))\D(h)$$
is  bijective. For any $a\in A$, we have
\begin{eqnarray*}
aF(f) &=& a((id\o S(f))\D(h))\\
      &=& ((id\o S(f))((a\o 1)\D(h))\\
      &=&  ((id\o S(f))((1\o S^{-1}(a))\D(h))\\
      &=& (id\o S(af))\D(h)\\
      &=& F(af)
\end{eqnarray*}
 So $F$ is a module isomorphism, then  $A$ is Frobenius.
\end{proof}

Note that when $A$ is a finite dimensional weak Hopf algebra, $A$ has a  faithful cointegral if and only if $A$ is a Frobenius algebra (see the Theorem 3.16 in \cite{BNS}).

\begin{proposition}
Let $A$ be a regular weak multiplier Hopf algebra. If $A$ is Frobenius and $\v\in \widetilde{A}$, then $A$ has a cointegral.
\end{proposition}
\begin{proof}
Let $F$ be the isomorphism. Then we have
$$aF(\v)= F(a\v)=F(\v(\cdot a))=F(\v(\cdot \v_t(a)))=\v_t(a)F(\v).$$
So $F(\v)$ is a cointegral.
\end{proof}

Remark that we do not have $\v\in \widetilde{A}$ in general. But for regular multiplier Hopf algebras and weak Hopf algebras, we always have $\v\in \widetilde{A}$.

Next let us recall Lemma 3.5 in \cite{VDZ}. Assume that $A$ is a regular weak multiplier Hopf algebra.

\begin{lemma}\label{bot}
Let $A$ be Frobenius, the isomorphism from $\widetilde{A}$ to $A$ is denoted by $F$. For any left ideal $I$ of $A$ we have $F(I^{\bot})=r(I)$ where $I^{\bot}=\{f\in \widetilde{A}|f(a)=0, \forall a\in A\}$.
\end{lemma}

Although we are working with regular weak multiplier Hopf algebras, the proof here is the same with the one in \cite{VDZ}.

\begin{proposition}
Let $A$ be a regular weak multiplier Hopf algebra with a faithful cointegral $h$, then every left proper left ideal has a non-zero right annihilator. Moreover,  every proper two-sided  ideal has a non-zero right annihilator.
\end{proposition}
\begin{proof}
 By the lemma above, we only need to show that  $I^{\bot}$ is non-trivial when $I$ is a proper left ideal. Now assume that  $I$ is a proper left ideal. Choose $f\in A'$ so that $f\neq 0$ but $f|I=0$. For any $a\in A$ and let $g=f(a\cdot)$, so $g\in I^{\bot}$. We claim that $g\neq 0$ then we finish the proof. In fact, if $g=0$ for all $a$, then we have $f(ab)=0$ for any $a$ and $b$. This means that $f=0$, contradiction.
\end{proof}

\begin{proposition}
Let $A$ be a regular weak multiplier Hopf algebra. Denote $I=Ker \v$. If there exists a non-zero element $h\in A$ such that $ah=0$ for any $a\in I$, then $A$ has  a cointegral.
\end{proposition}
\begin{proof}
Note that $ba-b\v_t(a)\in I$ for any $a, b\in A$. So we have $(ba-b\v_t(a))h=0$ for all $b$ and $a$. Cancel $b$ then we can get $ah=\v_t(a)h$, i.e. $k$ is  a left cointegral.
\end{proof}

Remark that the kernel of the counit is not a ideal in the case of weak multiplier Hopf algebras. The condition in this proposition also means that the kernel of $\v$ has a non-zero right annihilator.

\begin{proposition}
Let $A$ be a regular weak multiplier Hopf algebra with a faithful cointegral $h$, then $A$ is quasi-Frobenius.
\end{proposition}

\begin{proof}
Assume that $h$ is the faithful cointegral and $F$ is the module isomorphism from $\widetilde{A}$ to $A$ which is defined as $F(f)=(id\o S(f))\D(h)$. Let $I$ be a left ideal of $A$. We have to show $I=lr(I)$. The inclusion $I\subseteq lr(I)$ is obviously. If $I\neq lr(I)$, then we can find $f\in A'$ such that $f$ is zero on $I$ but not zero on $lr(I)$. Assume that $f(a)\neq 0$ for some $a\in lr(I)$. Now take $e$ so that $ea=a$, then we can define  $g=f(e\cdot)$ such that $g\in I^{\bot}$ and $g$ is not zero on $lr(I)$. Now $g\in \widetilde{A}$. For any $x\in lr(I)$, by the Lemma \ref{bot}, we have $xF(g)=0$. It means that $g$ is zero on $lr(I)$. Contradiction.
\end{proof}

Conversely, if $A$ is quasi-Frobenius, what can we get? We except to get that $A$ is discrete, but this is not true here.

\begin{remark}
Note that when $A$ is regular multiplier Hopf algebra, we have three equivalent descriptions, see Theorem 3.3, 3.6, and 3.8 in \cite{VDZ}. As we see now, they are not equivalent relations any more.

 Weak multiplier Hopf algebra of discrete type is  a regular weak multiplier Hopf algebra with a faithful set of cointegral. In this section we only consider the special case of a single faithful cointegral. We still have some examples which do not have a single faithful integral. So it is expected to have  a more general theory which could include more examples.

\end{remark}

For multiplier Hopf algebras, we have compact type and discrete type.  Let $A$ be a regular multiplier Hopf algebra with integrals, we call $A$ of compact type if $A$ has an identity, we call $A$ of discrete type if there is  a cointegral. And we have a duality between them. If $A$ is discrete, then the dual $\widehat{A}$ is compact. Conversely, if $A$ is compact, then $\widehat{A}$ is discrete. These results can be found in \cite{VD1}.

Now we will generalize the above results to weak multiplier Hopf algebras.

\begin{definition}
Let $(A, \D)$ be a regular weak multiplier Hopf algebra with a faithful set of integrals (Now we will call $A$ an algebraic quantum groupoid). We call $A$ of compact type if $A$ has an identity.
\end{definition}

\begin{proposition}
If $A$ is compact, then the dual $\widehat{A}$ is discrete.
\end{proposition}
\begin{proof}

Suppose that $A$ is compact. Let $\int$ be the faithful set of integrals on $A$. Note that now $A$ is a weak Hopf algebra and $\int\subset \widehat{A}$. In fact for any $\omega\in \widehat{A}, a\in A$, we have
\begin{eqnarray*}
\<a,\widehat{\v_t}(\omega)\vp\> &=& \<\omega\rhd a, \widehat{\v_t}(\omega)\>\\
                            &=& \<\v_t(a_{(1)}\vp(a_{(2)})), \omega\>\\
                            &=& \<a_{(1)}\vp(a_{(2)}), \omega\>\\
                            &=& \omega(a_{(1)})\vp(a_{(2)})\\
                            &=& (\omega\vp)(a)
\end{eqnarray*}
In the third equality, we use the fact that $(id\o \vp)\D(a)$ belongs to the target algebra. Next we will show that the  $\int$ is a faithful set of cointegral.
For any $\vp\in \int, \omega=\vp'(a\cdot )\in \widehat{A}$, we have
$$\widehat{\D}(\vp)(1\o \vp')=\sum \vp(\cdot S(a_{(1)}))\o \vp'(a_{(2)}\cdot ).$$
Since every element in $A$ is a linear combination of elements of the form $\sum S(a_{(1)})\vp(a_{(2)}c)$ where $a,c\in A$. So the left leg of $\D(\int)$ is all of $\widehat{A}$. Similarly for the right leg.
\end{proof}

\begin{proposition}
Let $A$ be an algebraic quantum groupoid. If $A$ has a single faithful cointegral, then $\widehat{A}$ is compact.
\end{proposition}
\begin{proof}
Let $h$ be the the faithful cointegral.  Then we can get the integral $\vp_h$. Now by  Proposition \ref{counit} we can get  $\v \in \widehat{A}$, so the dual is  compact.
\end{proof}

Especially, if $A$ has a faithful integral $\vp$ and a faithful cointegral $h$, then the dual $\widehat{A}$ is compact. Let $\vp_h$ be the integral so that $(id\o \vp_h)\D(h)=1$. From the Proposition 1.8 in \cite{VW4} we know that there is an element $y\in A_s$ so that $\vp_h(a)=\vp(ay)$ for all $a$. Now we have
$$\v(a)=\vp_h(ah)=\vp(ahy),$$
so $\v\in \widehat{A}.$

\begin{proposition}\label{f.d.}
If $A$ is both of compact type and of discrete type, then $A$ is finite-dimensional.
\end{proposition}
\begin{proof}
By Proposition \ref{multiplier}, there exists  cointegral $h_1$ and  $\omega_1\in A'$ such that $(id\o \omega_1)\D(h_1)=1$. For any $a\in A$, we  have
$$(a\o 1)\D(h_1)=(1\o S^{-1}(a))\D(h_1).$$
 Applying $id\o \omega_1$ to the equality then we can get
$$a=\sum h_{1(1)}\omega_1(S^{-1}(a)h_{1(2)}).$$
This shows that $A$ is spanned by only finitely many elements, then $A$ is finite dimensional.
\end{proof}

It is easy to obtain the following results.

\begin{corollary}
Let $A$ be a regular weak multiplier Hopf algebra. If $A$ is both compact and discrete, then $A$  is a finite dimensional weak Hopf algebra.
\end{corollary}

\section{Examples and applications}
First note that there is a relation between the integrals on the dual and the cointegrals in the original weak multiplier Hopf algebra.
\begin{proposition}
Let $A$ be an algebraic quantum groupoid. If $h$ is a cointegral in $A$, then $h$ is a left integral on $\widehat{A}$.
\end{proposition}
\begin{proof}
We are going to prove an equivalent statement. Assume that $g$ is a cointegral in $\widehat{A}$, now we will show that $\vp$ is a left integral on $A$. For any $f\in \widehat{A}, a\in A$, we have $(fg)(a)=(\v_t(f)g)(a)$. Then $f(a_{(1)}g(a_{(2)}))=f(\v_t(a_{(1)})g(a_{(2)}))$, so $(id\o g)\D(a)\in A_t$. Note that any element in $\widehat{A}$ is of the form $\vp(b\cdot)$ where $\vp$ is a integral and $b\in A$, then $a_{(1)}$ can be covered by $b$.
\end{proof}

\begin{remark}
 Conversely, this is not true. If $\vp$ is the integral on $A$, in general we do not have $\vp\in \widehat{A}$. When $A$ is a finite  Hopf algebra, then we have $\vp\in \widehat{A}$. Note that in this case $\widehat{A}$ is equal to $A'$. When $A$ is a regular multiplier Hopf algebra without unit, then $\vp\notin \widehat{A}$.
\end{remark}

 Now let us consider two basic examples associated with a groupoid.
\begin{example}
Let $G$ be any groupoid. Let $A$ be the weak multiplier Hopf algebra as introduced in Example \ref{example1}. It is the space of complex functions on $G$ with finite support and pointwise product. Recall that $A_s$ and $A_t$ are the algebra of all complex functions on the units.  As we see, the element $h$ is a left cointegral if and only if $h$ has finite support in the units. And in this case the left integral and right integrals coincide. Let $H$ be the space of  all left cointegrals. It is easy to find that $H$ is a faithful space, so $K(G)$ is discrete. In fact, for any $h\in H, y\in A_s$, if $hy=0$ then we must have $y=0$, otherwise $h$ is zero.

 We can use this example to illustrate the map $\g$ in Proposition \ref{gamma}. For any $p\in G$, we have
 \begin{eqnarray*}
 \g(h)(p)&=& h_{(1)}(p)S(\v_s(h_{(2)}))(p)\\
         &=& h_{(1)}(p)\v_s(h_{(2)})(p^{-1})\\
         &=& h_{(1)}(p)h_{(2)}(pp^{-1})\\
         &=& \D(h)(p, pp^{-1})\\
         &=& h(p).
 \end{eqnarray*}
So now we have $\g=id$.

We can also consider the dual of $A$. It is the algebra $\mathbb{C}G$ of complex functions with finite support on $G$ with convolution product. We have a canonical imbedding $p\mapsto \l_p$ from $G$ to $\mathbb{C}G$. The target and source maps are given by $\v_t(\l_p)=\l_e$ where $e=t(p)$ and $\v_s(\l_p)=\l_e$ where $e=s(p)$ for any $p\in G$. Here $t(p)$ is used for the target of $p$ and $s(p)$ denotes the source of $p$. For any $p, q\in G$, we have $\l_p\l_q=\l_{pq}$ when $pq$ is defined, otherwise is $0$. For any $g\in \widehat{A}_s, f\in A$, consider the map $\vp: f\mapsto \sum_ug(u)f(u)$. Here we take the sum over all elements $u\in G$ and this is well-defined because $f$ has finite support. If $\vp\in \mathbb{C}G$, then it is a left cointegral in $\mathbb{C}G$. And  any cointegral in $\mathbb{C}G$ is of this form. Take $\l_p\in \mathbb{C}G$, we have
$$(\l_p\vp)(f)=\sum_u g(u)f(pu).$$
Now we only take the sum over elements $u$ with the property $t(p)=s(u)$. On the other hand we have
$$(\v_t(\l_p)\vp)(f)=(\l_e\vp)(f)=\sum_u g(u)f(u).$$
Now $u$ must satisfy $s(u)=e=t(p)$. So $\vp$ is a left cointegral indeed. If we take $g=1$, then $\vp_1: f\mapsto \sum_uf(u)$ is a faithful cointegral. We also refer to Section 3 in \cite{VW4} where the integrals on $K(G)$ and $\mathbb{C}G$ were studied.

\end{example}

{\it The special cases}\newline

First let us consider the two special cases: multiplier Hopf algebras and weak Hopf algebras. They both had been studied extensively in the literature. We will recall some results here. On the other hand, our theory cover the following results.

Multiplier Hopf algebras of discrete type has been studied in \cite{VDZ}. Let $A$ be a regular multiplier Hopf algebra. If there exists a cointegral $h$ in $A$, then the cointegral is unique and faithful. Then we can construct the corresponding integral $\vp_h$ on $A$.  By Proposition 3.4 in \cite{VD1}, we know that every integral on $A$ is faithful. And $A$ is of discrete type if and only if $A$ is quasi-Frobenius, see Theorem 3.8 of \cite{VDZ}.

Assume that $A$ is a finite dimensional weak Hopf algebra. From Theorem 3.16 in \cite{BNS} we know that  there exist faithful cointegrals in $A$ if and only if $A$ is Frobenius or $\widehat{A}$ is Frobenius. Moreover, for any infinite dimensional weak Hopf algebra $A$, we find that if $A$ has a faithful cointegral then $A$ must be a   Frobenius algebra. We also present another method to study  cointegrals in weak Hopf algebra in this paper.  Note that the faithful cointegral is called non-degenerate integral in the paper \cite{BNS}.\\

In paper \cite{VD4}, Van Daele investigated the separability idempotents for the algebras that are not finite dimensional and are not unital. One of examples is the discrete quantum group. By this he means a multiplier Hopf algebra with a cointegral $h$ so that $\v(h)=1$ where $\v$ is the counit. Then $\D(h)$ is a regular separability idempotent in $M(A\o A)$. Now we can generalize it to weak multiplier Hopf algebra. First let us recall the definition of separability idempotent.

 Let  $B$ and $C$ be two non-degenerate algebras and $E$ be an idempotent element in $M(B\o C)$. We always require  that
 $$
 E(1\o c)\in B\o C \quad \text{and} \quad (b\o 1)E\in B\o C
 $$
 for all $b\in B$ and $c\in C$. $E$ is regular if we also require $(1\o c)E\in B\o c$ and $E(b\o 1)\in B\o C$.

 By the left leg and the right leg of $E$ we mean the smallest subspaces $V$ of $B$ and $W$ of $C$ satisfying
 $$
 E(1\o c)\in V\o C \quad \text{and} \quad (b\o 1)E\in B\o W
 $$
 for all $c\in C$ and $b\in B$. $E$ is called full if the left leg and right leg of $E$ are all of $B$ and $C$ respectively.

\begin{definition}
Let $E$ be a full idempotent in $M(B\o C)$. If there are non-degenerate anti-homomorphisms $S_B: B\ra M(C)$ and $S_C: C\ra M(B)$ such that
$$
E(b\o 1)=E(1\o S_B(b)) \quad \text{and} \quad (1\o c)E=(S_C(c)\o 1)E
$$
for all $b\in B$ and $c\in C$. Then we call $E$ a separability idempotent in $M(B\o C)$.
\end{definition}

\begin{proposition}
Let $A$ be a regular weak multiplier Hopf algebra with a faithful cointegral $h$. If $h$ is idempotent, then $\D(h)$ is a regular separability idempotent in $M(A\o A)$.
\end{proposition}
\begin{proof}
Since we assume $h$ is idempotent, then $\v_t(h)=1$ and $\v_s(h)=1$, so $h$ is also a right cointegral. In general the target and source algebras do not contain the unit $1$, but we do have some examples which contain 1. Now it follows that $\D(h)$ is an idempotent in $M(A\o A)$. From the theory in this paper, we know that $\D(h)$ is full and the antipode $S$ gives the non-degenerate anti-homomorphisms $S_B$ and $S_C$. Because $A$ is regular, then we have a regular separability idempotent.
\end{proof}

 Let $E$ be a separability idempotent in $ M(B\o C)$, then we can make $A=C\o B$ into a weak multiplier Hopf algebra. If $E$ is regular, then $A$ is regular. The construction is given by Van Daele in \cite{VD4,VD5,VW4}. Now we will investigate the cointegrals in this example.

\begin{example}
Let $A$ be the regular weak multiplier Hopf algebra associated with a regular separability idempotent $E$. $\vp_B$ and $\vp_C$ are the distinguished linear functionals on $B$ and $C$. The coproduct is given by $\D(c\o b)=c\o E\o b$ for any $c\in C, b\in B$. The counit $\v(c\o b)=\vp_B(S_C(c)b)=\vp_C(cS_B(b))$. The source and target maps are $\v_s(c\o b)=1\o S_c(c)b$ and $\v_t(c\o b)=cS_B(b)\o 1$. Then $c\o b$ is a cointegral if and only if $(c'c\o b'b)=c'S_B(b')c\o b$ for all $c'\in C, b'\in B$.

We can say more about the dual weak multiplier Hopf algebra $\widehat{A}$. Note that we can look at $A$ as the algebra spanned by elements of the form $cb$ and $cb=bc$. Then $\widehat{A}$ can be identified with the space $B\o C$ via the pairing
$$\<cb, u\o v\>=\vp_B(bS_C(v))\vp_C(S_B(u)c),$$
for any $u\in B, v\in C$. And we will use $u\diamond v$ for $u\o v$ when it is an element in $\widehat{A}$. Now we have
$(u\o v)(u'\o v')=\v(vu')u\diamond v'$
and
\begin{eqnarray*}
\widehat{\v_t}(u\diamond v)(u'\diamond v') &=&\vp_C(v)\v(E_2u')uE_1\diamond v'\\
                                           &=& \vp_C(v)uu'\diamond v'.
\end{eqnarray*}
In the second equality we use the formula $(E_1\diamond v)\v(E_2u)=u\diamond v$ which was proved in the Remark 3.4 of \cite{VD5}.
 If $u'\o v'$ is a left cointegral in $B\o C$, then $\v(vu')u\diamond v'=\vp_C(v)uu'\diamond v'$ for any $u, v$. But this is not possible, so there is no cointegral in $\widehat{A}$.
\end{example}

\begin{center}
 {\bf Acknowledgement}
\end{center}

I would like to thank Alfons Van Daele for providing  material of his work on the subject and for discussions about this work.

The work was partially supported by the NNSF of China (No. 11601231), the Fundamental Research Fund for the Central Universities (No. KJQN201716).

\end{document}